\documentclass[12pt]{amsart}

  \usepackage{a4wide}
  \usepackage{amssymb}
  \usepackage{amsmath}
  \usepackage{amsthm}
  \usepackage{amsfonts}
  \usepackage{amscd}
  \usepackage{xspace}
  \usepackage{boxedminipage}
  \usepackage{nomencl}
    
    \renewcommand{\nomgroup}[1]{\medskip \item {\bf #1}}
  \usepackage{booktabs}
  \usepackage{multirow}
  \usepackage[figuresright]{rotating}
  \usepackage[format=hang,font=small]{caption} % change how captions for figures are displayed
  \usepackage[curve,graph]{xypic}
  \usepackage{lscape}
  \usepackage{subfigure}

  \usepackage{stmaryrd}

  \RequirePackage{ifthen}
  \renewcommand{\nomgroup}[1]{%
  \ifthenelse{\equal{#1}{A}}{\item[\textbf{Algebra}]}{%
  \ifthenelse{\equal{#1}{Co}}{\item[\textbf{Combinatorics}]}{%
  \ifthenelse{\equal{#1}{A}}{\item[\textbf{Arrangements}]}{%
  {}}}}}

  \usepackage{tikz}
  \usetikzlibrary{calc}

  \makeatletter
  \CheckCommand*\refstepcounter[1]{\stepcounter{#1}%
      \protected@edef\@currentlabel
       {\csname p@#1\endcsname\csname the#1\endcsname}%
  }
  \renewcommand*\refstepcounter[1]{\stepcounter{#1}%
    \protected@edef\@currentlabel
      {\csname p@#1\expandafter\endcsname\csname the#1\endcsname}%
  }
  \def\labelformat#1{\expandafter\def\csname p@#1\endcsname##1}
  \DeclareRobustCommand\Ref[1]{\protected@edef\@tempa{\ref{#1}}%
     \expandafter\MakeUppercase\@tempa
  }
  \makeatother

  \makeatletter
  \newcommand{\numberlike}[2]{%
     \expandafter\def\csname c@#1\endcsname{%
         \expandafter\csname c@#2\endcsname}%
  }
  \makeatother

  \def\DefaultNumberTheoremWithin{section}

  \theoremstyle{plain}
  \newtheorem{Lemma}{Lemma}
     \numberwithin{Lemma}{\DefaultNumberTheoremWithin}
     \labelformat{Lemma}{Lemma~#1}
  
     \numberwithin{Claim}{\DefaultNumberTheoremWithin}
     \numberlike{Claim}{Lemma}
     \labelformat{Claim}{Claim~#1}
  \newtheorem{Theorem}{Theorem}
     \numberwithin{Theorem}{\DefaultNumberTheoremWithin}
     \numberlike{Theorem}{Lemma}
     \labelformat{Theorem}{Theorem~#1}
  \newtheorem{Corollary}{Corollary}
     \numberwithin{Corollary}{\DefaultNumberTheoremWithin}
     \numberlike{Corollary}{Lemma}
     \labelformat{Corollary}{Corollary~#1}
  \newtheorem{Proposition}{Proposition}
     \numberwithin{Proposition}{\DefaultNumberTheoremWithin}
     \numberlike{Proposition}{Lemma}
     \labelformat{Proposition}{Proposition~#1}
  
     \numberwithin{Conjecture}{\DefaultNumberTheoremWithin}
     \numberlike{Conjecture}{Lemma}
     \labelformat{Conjecture}{Conjecture~#1}

  \theoremstyle{definition}
  
     \numberwithin{Definition}{\DefaultNumberTheoremWithin}
     \numberlike{Definition}{Lemma}
     \labelformat{Definition}{Definition~#1}

  \theoremstyle{definition}
  
     \numberwithin{Question}{\DefaultNumberTheoremWithin}
     \numberlike{Question}{Lemma}
     \labelformat{Question}{Question~#1}

  \theoremstyle{definition}
  
     \numberwithin{Problem}{\DefaultNumberTheoremWithin}
     \numberlike{Problem}{Lemma}
     \labelformat{Problem}{Problem~#1}

  \theoremstyle{remark}
  
     \numberwithin{Remark}{\DefaultNumberTheoremWithin}
     \numberlike{Remark}{Lemma}
     \labelformat{Remark}{Remark~#1}
  \theoremstyle{remark}

     \numberwithin{Example}{\DefaultNumberTheoremWithin}
     \numberlike{Example}{Lemma}
     \labelformat{Example}{Example~#1}
  
     \labelformat{Case}{Case~#1}
     \numberwithin{Case}{Lemma}
  
     \labelformat{Step}{Step~#1}
     \numberwithin{Step}{Lemma}

  \labelformat{equation}{(#1)}
  \labelformat{figure}{Figure~#1}
  \labelformat{chapter}{Chapter~#1}
  \labelformat{appendix}{Appendix~#1}
  \labelformat{section}{\S~#1}
  \labelformat{subsection}{\S~#1}
  \labelformat{subsubsection}{\S~#1}

  \def\eqref{\ref}

  \allowdisplaybreaks
  \newcommand{\mb}{\mathbb}

  \def\binomial(#1,#2){{#1\choose #2}}
\def\delplus(#1,#2,#3){\ensuremath{\pi}_{#1,#2,#3}}
\def\delminus(#1,#2){\ensuremath{\operatorname{\partial}}_{#1,#2}}
  \def\redword(#1){{\ensuremath{\tilde{#1}}}}

  \newcommand{\Tor}{\ensuremath{\operatorname{Tor}}}
  \newcommand{\dual}{{\circ}}
  \newcommand{\lcm}{\ensuremath{\operatorname{lcm}}}
  \newcommand{\homology}{\ensuremath{\operatorname{H}}}
  \newcommand{\pol}{\mathrm{pol}}
  \newcommand{\Supp}{\mathrm{supp}}

  \newcommand{\RR}{\mb{R}}
  
  \newcommand{\KK}{\mb{K}}

  \newcommand{\xx}{\mathbf{x}}

  \newcommand{\pp}{\mathfrak{p}}

  \newenvironment{note}[1][Note]
   {\bigskip\begin{center}\begin{boxedminipage}{4.5in}\setlength{\parindent}{1em}\noindent\textbf{#1. }}
   {\end{boxedminipage}\end{center}\bigskip}

\begin{document}

% ------------------------------------------------------------------------

\title[The Golod property for products and high symbolic powers]{The Golod property for products and high symbolic powers of monomial ideals}

% ------------------------------------------------------------------------

\author[S. A. Seyed Fakhari]{S. A. Seyed Fakhari}

\address{Department of Mathematical Sciences,
Sharif University of Technology, P.O. Box 11155-9415, Tehran, Iran.}

\email{fakhari@ipm.ir}

\author[V. Welker]{V. Welker}

\address{Fachbereich Mathematik und Informatik, Philipps-Universit\"at Marburg,
    35032 Marburg, Germany}

\email{welker@mathematik.uni-marburg.de}

%\urladdr{http://math.ipm.ac.ir/fakhari/}

% ------------------------------------------------------------------------

\begin{abstract}
  We show that for any two proper monomial ideals $I$ and $J$ in the polynomial ring $S = k[x_1,\ldots, x_n]$ 
  the ring $S/IJ$ is Golod. We also show that if $I$ is squarefree then for large enough $k$ the quotient $S/I^{(k)}$ of $S$ 
  by the $k$\textsuperscript{th} symbolic power of $I$ is Golod. As an application we prove that the multiplication on the cohomology algebra of
  some classes of moment-angle complexes is trivial. 
\end{abstract}

% ------------------------------------------------------------------------

\subjclass[2000]{05E45, 05D05}

% ------------------------------------------------------------------------

\keywords{Monomial ideal, Golod ring, product ideal, symbolic power, strong gcd-condition, moment-angle complex}

% ------------------------------------------------------------------------

\thanks{The first author acknowledges support from DAAD as a visiting PhD student at Philipps-Universit\"at Marburg}

% ------------------------------------------------------------------------

\maketitle

%%%%%%%%%%%%%%%%%%%%%%%%%%%%%%%%%%%%%%%%%%%%%%%%%%%%%%%%%%%%%%%%%%%%%%%%%%

\section{Introduction} \label{sec:intro}

  For a graded ideal $I$ in the polynomial ring $S = \KK[x_1,\ldots,x_n]$ in $n$ variables over the field $\KK$ the ring $S/I$ is called {\it Golod}
  if all Massey operations on the Koszul complex of $S/I$ with respect of $\xx = x_1,\ldots, x_n$ vanish. 
  The naming gives credit to Golod \cite{Golod}
  who showed that the vanishing of the Massey operations is equivalent to the equality case in the following coefficientwise inequality of power-series which was 
  first derived by Serre:
  \begin{eqnarray*}
     \sum_{i \geq 0} \dim_\KK \Tor_i^{S/I} (\KK,\KK) t^i  \leq \frac{(1+t)^n}{1-t \sum_{i \geq} \dim_\KK \Tor_i^S(S/I,\KK) t^i} 
  \end{eqnarray*}

  We refer the reader to \cite{Avramov} and \cite{GulliksenLevin} for further information on the Golod property
  and to \cite{BrunsHerzog} and \cite{HerzogHibi} for the basic concepts from commutative algebra underlying this paper. 
  We prove the following two results.

  \begin{Theorem}
     \label{thm:prod}
     Let $I$, $J$ be two monomial ideals in $S$ different from $S$. Then $S/IJ$ is Golod.
  \end{Theorem}

  In the statement of the second result we write $I^{(k)}$ for the $k$\textsuperscript{th} symbolic power
  of the ideal $I$.

  \begin{Theorem} \label{thm:sym}
    Let $I$ be a squarefree monomial ideal in $S$ different from $S$. Then for $k \gg 0$ the $k$\textsuperscript{th} symbolic power 
    $I^{(k)}$ is Golod for $k\gg 0$.
  \end{Theorem}

  Besides the strong algebraic implications of Golodness the case of squarefree monomial ideals relates to
  interesting topology. Let $\Delta$ be a simplicial complex on ground set $[n]$ and let $\KK[\Delta]$ be its 
  Stanley--Reisner ring (see \ref{sec:mom} for basic facts about Stanley-Reisner rings).
  By work of Buchstaber and Panov \cite[Thm. 7.7]{BuchstaberPanov}, extending an additive isomorphism from
  \cite{GasharovPeevaWelker2}, it is known that there is an algebra
  isomorphism of the Koszul homology $\homology_*(\xx,k[\Delta])$ and the singular cohomology ring $\homology^*(M_\Delta;k)$
  where $M_\Delta = \{ (v_1,\ldots, v_n) \in (D^2)^n ~|~\{ i~|~v_i \not\in S^1 \} \in \Delta \}$.
  Here $D^2 = \{ v \in \RR^2 ~|~||v|| \leq 1 \}$ is the unit disk in $\RR^2$ and $S^1$ its bounding unit circle.
  Note that the isomorphism ist not graded for the usual grading of $\homology_*(\xx,k[\Delta])$ and $\homology^*(M_\Delta;k)$.
  The complex $M_\Delta$ is the {\it moment-angle complex} or {\it polyhedral product} of the pair $(D^2,S^1)$ for $\Delta$ (we refer the reader
  to \cite{BuchstaberPanov} and \cite{DenhamSuciu} for background information). 
  Last we write $\Delta^\dual = \{ A \subseteq [n] ~|~[n] \setminus A \not\in \Delta \}$ for the 
  {\it Alexander dual} of the simplicial complex $\Delta$. Now we are in position to formulate the following
  consequence of \ref{thm:prod}.

  \begin{Corollary}
    \label{cor:join}
    Let $\Delta$ be a simplicial complex such that 
    $\Delta = (\Delta_1^\dual * \Delta_2^\dual)^\dual$ for two simplicial complexes
    $\Delta_1$, $\Delta_2$ on disjoint ground sets. 
    Then the multiplication on $\homology^*(M_\Delta;k)$ is trivial.
  \end{Corollary}

  The main tool for the proof of \ref{thm:prod} and \ref{thm:sym} is combinatorial. Let $I$
  be a monomial ideal and write $G(I)$ for the set of minimal monomial generators of $I$.
  In \cite[Def. 3.8]{Joellenbeck} the author introduces a combinatorial condition on $G(I)$ that 
  in \cite{BerglundJoellenbeck} was shown to imply the Golod property for $S/I$. 
  The ideal $I$ is said to satisfy the {\it strong gcd--condition} 
  if there exists a linear order $\prec$ on $G(I)$ such that for any two monomials $u \prec v\in G(I)$ 
  with $\gcd(u , v) = 1$ there exists a monomial $w \in G(I)$ with $w \neq u,v$ such that $u\prec w$ and $w$
  divides $\lcm(u,v) = uv$. The following result from \cite{BerglundJoellenbeck} removes an
  unnecessary assumption from the statement of Theorem 7.5 in \cite{Joellenbeck}. 

  \begin{Theorem}[Thm. 5.5 \cite{BerglundJoellenbeck}] 
    \label{thm:crit}
    Let $I$ be a monomial ideal. 
    If $I$ satisfies the strong gcd-condition then $S/I$ is Golod.
  \end{Theorem}

  We refer the reader to \cite{Berglund} for the relation of the gcd--condition to standard
  combinatorial properties of simplicial complexes.

  The paper is organized as follows.
  In \ref{sec:prod} we verify the strong gcd--condition for products of monomial ideals
  and in \ref{sec:sym} for high symbolic powers of squarefree monomial ideals. This yields
  \ref{thm:prod} and \ref{thm:sym}.
  In \ref{sec:mom} we study the implications of \ref{thm:prod} and \ref{thm:sym} 
  on moment--angle complexes. In particular, we derive
  \ref{cor:join}.

%%%%%%%%%%%%%%%%%%%%%%%%%%%%%%%%%%%%%%%%%%%%%%%%%%%%%%%%%%%%%%%%%%%%%%%%%%

\section{Product of monomial ideals} \label{sec:prod}

  Since by \ref{thm:crit} for any monomial ideal $I$ satisfying the strong gcd-condition 
  the ring $S/I$ is Golod, the following result immediately implies \ref{thm:prod}.

  \begin{Proposition} \label{prop:prod}
    For any two monomial ideals $I,J$ in $S$ that are different from $S$ 
    the ideal $IJ$ satisfies strong-gcd condition.
  \end{Proposition}
  \begin{proof}
    Fix a monomial order $<$ on the set of monomials of $S$. 
    We define a linear order $\prec$ on $G(IJ)$ as follows: For two monomials $u, v \in G(IJ)$, we set $u\prec v$ 
    if and only if $\deg (u) > \deg(v)$ or $\deg(u) = \deg(v)$ and $u < v$. 
    Now the following claim states that $G(IJ)$ ordered by $\prec$ satisfies the strong gcd--condition. 

    \noindent {\sf Claim:} For any two monomials $u, v\in G(IJ)$ with $u \prec v$ and $\gcd(u, v) = 1$ there exists a
    monomial $w\in G(IJ)$ different from $u$ and $v$ such that $u\prec w$ and $w$ divides $\lcm(u, v)=uv$. 
  
    Since $u , v\in G(IJ)$, there exist monomials $u_1,v_1\in G(I)$ and $u_2, v_2 \in G(J)$ such 
    that $u=u_1u_2$ and $v=v_1v_2$. We consider two cases:

    \noindent {\sf Case 1.} $\deg (u) = \deg(v)$. 

    By definition, $u <v$. Therefore, $u_1 < v_1$ or $u_2 < v_2$, since otherwise 
    $u=u_1u_2 \geq v_1v_2=v$, which is a contradiction. Without loss of generality we may assume that $u_1 < v_1$.

    If $\deg (u_1) \geq \deg(v_1)$, then define $w':= v_1u_2$. Now $w'\in IJ$, $\deg (w') \leq \deg (u)$ and $u < w'$. 
    If $w'\in G(IJ)$, then we set $w=w'$ and clearly  $u\prec w$ and $w$ divides $\lcm(u, v)=uv$. By $\gcd(u, v) = 1$,
     it follows that $u_2\neq v_2$ and therefore $w\neq v$. If $w'\not\in G(IJ)$, then there exists a 
    monomial $w\in G(IJ)$, such that $w | w'$. Now $w$ divides $\lcm(u, v)=uv$ and since $\deg (w) < \deg (w') \leq \deg (u) = \deg (v)$, 
    we conclude that $u \prec v \prec w$ and $w\neq v$.

    If $\deg (u_1) < \deg (v_1)$, then $\deg (u_2) > \deg (v_2)$, since $\deg (u) = \deg (v)$. We define $w':=u_1v_2$. 
    Now $w'\in IJ$ and therefore there exists a monomial, say $w\in G(IJ)$, such that $w | w'$. Then $w$ divides $\lcm(u, v)=uv$ 
    and since $\deg (w)\leq \deg (w') < \deg (u) = \deg(v)$, we conclude that $u\prec v \prec w$ and $w\neq v$.

    \noindent {\sf Case 2.} $\deg (u) \neq \deg (v)$. 
  
    By definition of $\prec$ we must have $\deg (u) > \deg (v)$. Hence either of $\deg (v_1) < \deg (u_1)$ or $\deg (v_2) < \deg (u_2)$ holds. 
    Without loss of generality we assume that $\deg (v_1) < \deg (u_1)$. We define $w'=v_1u_2$. Then $w'\in IJ$ and therefore 
    there exists a monomial, say $w\in G(IJ)$, such that $w$ divides $w'$. Thus $w$ divides $\lcm(u, v)=uv$ and 
    $u\prec w$, since $\deg (w)\leq \deg (w') < \deg(u)$. Also $w \neq v$, because otherwise $v$ divides $v_1u_2$ and since 
    $\gcd(u_1, v) = 1$, it implies that $v$ divides $v_1$, which is impossible.
  \end{proof}

  In \cite[Thm. 4.1]{HerzogWelkerYassemi} it is shown that for a homogeneous ideal $I \neq S$ in the polynomial ring $S$  
  the ring $S/I^k$ is Golod for $k\gg 0$. For monomial ideals \ref{thm:prod} allows a more precise
  statement of their result which follows immediately from \ref{prop:prod} and \ref{thm:crit}.

  \begin{Corollary} \label{cor:pow}
    Let $I$ be a monomial ideal in the polynomial ring $S$ different from $S$. 
    Then for $k \geq 2$ the ideal $I^k$ satisfies the strong gcd--condition and hence
    the ring $R=S/I^k$ is Golod.
  \end{Corollary}

  Since by \cite[Thm. 3.4.5]{BrunsHerzog} the Koszul homology $\homology_*(\xx,S/I)$ for a 
  Gorenstein quotient $S/I$ of $S$ is a Poincar\'e duality algebra, $S/I$ cannot be
  Golod unless $I$ is a principal ideal. Thus \ref{thm:prod} implies that a product
  of monomial ideal can be Gorenstein if and only if it is principal.
  This observation is not new and can be seen as a consequence of a 
  very general result by Huneke \cite{Huneke} who showed that in any unramified 
  regular local ring no Gorenstein ideal of height $\geq 2$ can be a product.
  
\section{Symbolic powers of squarefree monomial ideals}
  \label{sec:sym}

  For our results on symbolic powers we have to restrict ourselves to squarefree
  monomial ideals $I$. This is due to the fact that our proofs use that 
  only in this case in the primary decomposition $I=\pp_1\cap\cdots\cap\pp_r$ of $I$ every $\pp_i$ is an
  ideal of $S$ generated by a subset of the variables of $S$ \cite[Lem. 1.5.4]{HerzogHibi}. 
  Moreover, in this situation 
  for a positive integer $k$ the $k$\textsuperscript{th} symbolic power $I^{(k)}$ of $I$
  coincides with $\pp_1^k\cap\ldots\cap \pp_r^k$ \cite[Prop.1.4.4]{HerzogHibi}. 

  Now we are ready to prove that the high symbolic powers of squarefree monomial ideals  
  of $S$ fulfill the strong-gcd condition.

  \begin{Proposition} \label{prop:sym}
    Let $I$ be a squarefree monomial ideal in $S$ different from $S$. Then for $k \gg 0$ the $k$\textsuperscript{th} symbolic power 
    $I^{(k)}$ satisfies the strong gcd--condition.
  \end{Proposition}
  \begin{proof}
    By \cite[Prop. 1]{Lyubeznik} the ring $A=\bigoplus_{i=0}^{\infty}I^{(i)}$ is Noetherian and therefore is a finitely 
    generated $\KK$-algebra. Assume that the set $\{y_1, \ldots, y_m\}$ is a set of generators for the $\KK$-algebra $A$. 
    Following \cite{Cowsik}, since $A$ is finitely generated, there exists a natural number $c$ such that 
    $A_0=\bigoplus_{i=0}^{\infty}I^{(ci)}$ is a standard $\KK$-algebra and therefore a Noetherian ring. Note that the set
    $$\{y_1^{\ell_1}\ldots y_m^{\ell_m}~|~  0\leq \ell_1, \ldots, \ell_m \leq c-1\}$$
    is a system of generators for $A$ as an $A_0$-module. Assume that the degree of a generator of the $A_0$-module $A$ is at most $\alpha$.
    Since $A_0$ is standard $\KK$-algebra for every integer $k> \max \{\alpha , c \}$ we 
    have $I^{(k)}= I^{(c)}I^{(k-c)}$. Thus for every $k> \max \{\alpha , c \}$, 
    the ideal $I^{(k)}$ is the product of two monomial ideals. Therefore, by \ref{prop:prod} it satisfies the strong-gcd condition.
  \end{proof}
 
  The following corollary is an immediate consequence of \ref{thm:crit} and \ref{thm:sym}.

  \begin{Corollary}
    Let $I$ be a squarefree monomial ideal in the polynomial ring $S$, which is not a principal ideal. Then
     for $k \gg 0$ the ring $S/I^{(k)}$ is not Gorenstein.
  \end{Corollary}
 
  We do not know which $k$ suffices. Indeed, we do not have an example of a monomial ideal $I \neq S$ squarefree or not
  for which $I^{(2)}$ is not Golod.
 
\section{Moment--angle complexes}
  \label{sec:mom}
  
  First, we recall some basics from the theory of Stanley-Reisner ideals.
  Let $\Delta$ be a simplicial complex on ground set $[n]$.  A subset $N \subseteq [n]$ such that 
  $N \not\in \Delta$ and $N \setminus \{i\} \in \Delta$ for all $i \in N$ is called a {\it minimal   
  non-face} of $\Delta$. The {\it Stanley-Reisner ideal} $I_\Delta$ of $\Delta$ is
  the ideal in $S$ generated by the monomial $x_N$ for the minimal non-faces $N$ of $\Delta$ and
  the quotient $k[\Delta] = S/I_\Delta$ is Stanley-Reisner ring of $\Delta$.
  Indeed, the map  sending $\Delta$ to $I_\Delta$ is a bijection between the simplicial complexes on
  ground set $[n]$ and the squarefree monomial ideals in $S$. 
  To any monomial ideal $I$ its {\it polarization} $I^{\pol}$ \cite[p. 19]{HerzogHibi} is a 
  squarefree monomial ideal and it is known \cite[Thm. 3.5]{GasharovPeevaWelker} that $I$ is
  Golod if and only if $I^\pol$ is Golod. In addition, it follows from the vanishing of the
  Massey operations that the multiplication on Koszul homology $\homology_*(\xx,S/I)$ is trivial for
  all Golod $S/I$ -- by trivial we here mean that products of two elements of positive degree are $0$. 
  Thus the algebra isomorphism 
  of $\homology_*(\xx,k[\Delta])$ and $\homology^*(M_\Delta;k)$ \cite[Thm. 7.7]{BuchstaberPanov}
  together with \ref{thm:prod} and \ref{thm:sym} yields the following corollary.   
  Note that even though the isomorphism of $\homology_*(\xx,k[\Delta])$ and $\homology^*(M_\Delta;k)$
  is not graded in the usual grading it sends $\homology_0(\xx,k[\Delta])$ and $\homology^0(M_\Delta;k)$.
 
  \begin{Corollary} \label{cor:mom1}
    Let $I$ and $J$ be monomial ideals. Then for the simplicial complexes $\Gamma$ and $\Gamma^{(k)}$
    such that $(IJ)^\pol = I_{\Gamma}$ and $(I^{(k)})^\pol = I_{\Gamma^{(k)}}$ we have:
    \begin{itemize}
       \item[(i)] The multiplication on the cohomology algebra of $M_{\Gamma}$ is trivial. 
       \item[(ii)] The multiplication on the cohomology algebra of $M_{\Gamma^{(k)}}$ is trivial 
           for $k \gg 0$.
    \end{itemize}
  \end{Corollary}

  In general, the combinatorics and geometry of the simplicial complexes $\Gamma$ and $\Gamma^{(k)}$ cannot be
  easily controlled even in the case $I$ and $J$ are squarefree monomial ideals. Therefore, the result is 
  more useful in a situation when the ideals $IJ$ and $I^{(k)}$ themselves are squarefree monomial ideals. Since this 
  never happens for $I^{(k)}$ and $k \geq 2$ we confine ourselves to the case of products $IJ$.
  The following lemma shows that in this case $I$ and $J$ should be squarefree monomial ideals with generators in disjoint sets of 
  variables. Even though the lemma must be a known basic fact from the theory of the monomial ideals we did not 
  find a reference and hence for the sake of completeness we provide a proof. In the proof we denote for a monomial   
  $u$ by $\Supp(u)$ its ${\it support}$, which is the set of variables dividing $u$.
  
  \begin{Lemma} \label{lem:squa}
    Let $I, J$ be monomial ideals. Then $IJ$ is a squarefree monomial ideal if and only if $I$ and $J$ are
     squarefree monomial ideals such that $\gcd(u,v) = 1$ for all $u \in G(I)$ and $v \in G(J)$.
  \end{Lemma}
  \begin{proof}
     The ``if'' part of the lemma is trivial. The other direction states that if $IJ$ is
     a squarefree monomial ideal then for every $u \in G(I)$ and every $v \in G(J)$ the monomial $uv$ is squarefree. 
     Assume by contradiction that there exist monomials $u \in G(I)$ and $v \in G(J)$ such that $uv$ is not squarefree. 
     Among every $v$ with this property we choose one such that the set $\Supp (v)\setminus \Supp (u)$ has minimal cardinality. 
     Since $IJ$ is a squarefree monomial ideal, there exist squarefree monomials $u' \in G(I)$ and $v' \in G(J)$ such that 
     $\Supp (u') \cap \Supp (v') = \emptyset$ and $u'v'$ divides $uv$. 
     We distinguish two cases:

     \noindent {\sf Case:}  $\Supp (v)\setminus \Supp (u) = \emptyset$ 

     The assumption is equivalent to $\Supp (v)\subseteq \Supp (u)$. We conclude that $\Supp (u') \subseteq \Supp (uv) = \Supp (u)$.
     But then $u'$ divides $u$ and hence $u = u'$. But then $\Supp (uv) = \Supp (u')$ and $v' = 1$. This implies
     $J = S$ and $IJ = I$ which in turn shows that $I$ is a squarefree monomial ideal. Thus $u$ is squarefree and $v = v' = 1$.
     But then $uv$ is squarefree and we arrive at a contradiction.
     
     \noindent {\sf Case:}  $\Supp (v)\setminus \Supp (u) = \{x_{i_1}, \ldots x_{i_t}\}$ for some $t \geq 1$

     If $\Supp (v') \subseteq \Supp (v)$ then $v'$ divides $v$ and hence $v = v'$. Then $u'$ must divide $u$ and hence $u = u'$.
     But this contradicts the fact that $uv = u'v'$ is not squarefree. Hence $\Supp (u) \cap \Supp (v') \neq \emptyset$ and therefore $uv'$ is not squarefree. 
     Now assume that $\{x_{i_1}, \ldots, x_{i_t}\} \subseteq \Supp (v')$. Then since $\Supp (u') \cap \Supp (v') =\emptyset$, it follows that 
     $\Supp (u')\subseteq \Supp (u)$ and therefore $u'$ divides $u$. As above this yields a contradiction. 
     Thus $\{x_{i_1}, \ldots, x_{i_t}\} \nsubseteq \Supp (v')$ and the inclusion $\Supp (v') \subseteq \Supp (u) \cup \Supp (v)$ implies that the cardinality 
     of $\Supp (v')\setminus \Supp (u)$ is strictly less than the cardinality of $\Supp (v)\setminus \Supp (u)$, which contradicts the choice of $v$.
  \end{proof}
  
  Let us analyze the situation when $IJ$ is a squarefree monomial ideal more carefully. By \ref{lem:squa}, we may assume that $I = I_{\Delta}$ and $J = I_{\Delta'}$ 
  with simplicial complexes $\Delta$ and $\Delta'$
  that have a join decomposition $\Delta = 2^{V_1} * \Delta_1$ and $\Delta' = 2^{V_2} * \Delta_2$ for simplicial complexes $\Delta_1$ and $\Delta_2$ over disjoint ground
  sets and full simplices $2^{V_1}$ and $2^{V_2}$ over arbitrary finite sets $V_1$ and $V_2$. The product of the corresponding Stanley--Reisner ideals is easily 
  described. 

  \begin{Lemma}
    \label{lem:join}
    Let $\Delta_i$, $i=1,2$, be simplicial complexes on disjoint ground sets $\Omega_1$ and
    $\Omega_2$. Then $I_{\Delta_1} I_{\Delta_2} = I_{(\Delta_1^\dual * \Delta_2^\dual)^\dual}$.
  \end{Lemma}
  \begin{proof}
     Recall that by definition of the Alexander dual, for every simplicial complex $\Delta$ on the ground set $\Omega$, the maximal faces $A \in \Delta^\dual$ are those subsets 
     $A$ of $\Omega$ for which
     $\Omega \setminus A$ is a minimal non-face of $\Delta$. 
 
     The product $I_{\Delta_1} I_{\Delta_2}$ is generated by $x_{N_1}x_{N_2}$ for minimal non--faces
     $N_1$ and $N_2$ of $\Delta_1$ and $\Delta_2$ respectively. Hence it is generated
     by monomials corresponding to the union of $F_1 = \Omega_1 \setminus N_1$ and $F_2 = \Omega_2 \setminus N_2$ of 
     maximal faces of $\Delta_1^\dual$ and $\Delta_2^\dual$. Since for any maximal face $F$ of 
     $\Delta_1^\dual * \Delta_2^\dual$ there are maximal faces $F_1$ of $\Delta_1^\dual$ and
     $F_2$ of $\Delta_2^\dual$ such that $F = F_1 \cup F_2$ it follows that the monomials 
     $x_{N_1 \cup N_2} = x_{N_1} x_{N_2}$ for minimal non-faces $N_1$ of $\Delta_1$ and 
     $N_2$ of $\Delta_2$ are the generators $I_{(\Delta_1^\dual * \Delta_2^\dual)^\dual}$.
     Now the assertion follows.
  \end{proof}

  Now combining \ref{lem:join} and \ref{cor:mom1} (i) yields \ref{cor:join}.

  We note, that for the deduction of \ref{cor:join} one could have also argued using 
  \cite[Satz 2]{HerzogSteurich} instead of \ref{thm:prod}.

%%%%%%%%%%%%%%%%%%%%%%%%%%%%%%%%%%%%%%%%%%%%%%%%%%%%%%%%%%%%%%%%%%%%%%%%%%

%%%%%%%%%%%%%%%%%%%%%%%%%%%%%%%%%%%%%%%%%%%%%%%%%%%%%%%%%%%%%%%%%%%%%%%%%%

%%%%%%%%%%%%%%%%%%%%%%%%%%%%%%%%%%%%%%%%%%%%%%%%%%%%%%%%%%%%%%%%%%%%%%%%%%


\begin{thebibliography}{10}

\bibitem{Avramov} L.L. Avramov, Infinite free resolutions, in: J. Elias et al., 
  Six lectures on commutative algebra. {\it Prog. Math.} {\bf 166}, 1--118, Basel: Birkh{\"a}user, 1998.

%\bibitem {akm} L.L. Avramov, A.R. Kustin and M. Miller, Poincare series of modules over local rings
%of small embedding codepth or small linking number, {\it J. Algebra.} {\bf 118}, (1986) 162--204.

\bibitem{Berglund} A. Berglund, Poincare series of monomial rings, {\it J. Algebra.} {\bf 295}, (2006), 211-230.

%\bibitem {bbh} A. Berglund and J. Blasiak, P. Hersh, Combinatorics of multigraded Poincare series for
%monomial rings, {\it J. Algebra.} {\bf 308} (2007), 73--90.

\bibitem{BerglundJoellenbeck} A. Berglund and M. J{\"o}llenbeck, On the Golod property of Stanley--Reisner rings, {\it J. Algebra.} {\bf 315} (2007) 249--273.

\bibitem{BrunsHerzog} W. Bruns and J. Herzog, Cohen-Macaulay rings, Cambridge Stud. in Adv. Math. {\bf 39}, Cambridge: Cambridge 
  University Press, 1993.
 
\bibitem{BuchstaberPanov} V.M. Buchstaber and T. E. Panov, Torus actions and their applications in topology and combinatorics,
  Univ. Lect. Ser. {\bf 24}, Providence:Amer. Math. Soc., 2002.

%\bibitem {b1} M. Brodmann, The asymptotic nature of the analytic spread, {\it Math. Proc. Camb. Philos.
%Soc.} {\bf 86} (1979) 35--39.

\bibitem{Cowsik} R. C. Cowsik, Symbolic powers and the number of defining equations, In: Algebra and
  Its Applications, Lect. Notes in Pure and Appl. Math. {\bf 91}, 1985, 13--14.

\bibitem{DenhamSuciu} G. Denham, A. Suciu, Moment--angle complexes, monomial ideals, and {M}assey products,
    Pure and Appl. Math. Quart. {\bf 3} (2007) 25--60.

%\bibitem {cht} S.D. Cutkosky and J. Herzog and N.T. Trung, Asymptotic behaviour of the Castelnuovo--Mumford 
%  regularity, {\it Compositio Math.} {\bf 118} (1999) 243--261.

\bibitem{GulliksenLevin} T. H. Gulliksen and G. Levin, Homology of local rings, Queens Papers in Pure and Appl.
  Math. {\bf 20}, Kingston, Ontario: Queens University, 1969.

\bibitem{GasharovPeevaWelker} V. Gasharov, I. Peeva and V. Welker, The lcm-lattice in monomial resolutions, {\it Math.
Res. Lett.} {\bf 6} (1999), no. 5-6, 521--532.

\bibitem{GasharovPeevaWelker2} V. Gasharov, I. Peeva and V. Welker, Coordinate subspace arrangements and monomial ideals,
            T. Hibi (ed.), Computational commutative algebra and combinatorics, Tokyo: Math. Soc. of Jap.,
            Adv. Stud. Pure Math. {\bf 33}, 65--74, 2001.
      
\bibitem{Golod} E.S. Golod, On the homology of some local rings, Soviet Math. Dokl. {\bf 3} (1962) 745--748.

%\bibitem{hw} J. Herzog and V. Welker, The Betti polynomials of powers of an ideal, {\it J. Pure and Appl.
%Alg.} {\bf 215} (2011) 589--596.

\bibitem{HerzogHibi} J. Herzog, T. Hibi, Monomial ideals, Grad. Texts in Math {\bf 260}, Heidelberg:Springer, 2011. 

\bibitem{HerzogSteurich} J. Herzog, M. Steurich, Golodideale der Gestalt $\mathfrak{a} \cap \mathfrak{b}$, J. Algebra {\bf 58} (1979) 31--36.

\bibitem{Huneke} C. Huneke, Ideals defining {G}orenstein rings are (almost) never products, Proc. Amer. Math. Soc. {\bf 135} (2007) 
    2003--2005. 

\bibitem{HerzogWelkerYassemi} J. Herzog, V. Welker and S. Yassemi, 
  Homology of powers of ideals: Artin--Rees numbers of syzygies and the Golod property, Preprint 2011.

%\bibitem {k1} V. Kodiyalam, Asymptotic behaviour of Castelnuovo-Mumford regularity, {\it Proc. Amer.
%Math. Soc.} {\bf 128} (2000) 407--411.

%\bibitem {k} V. Kodiyalam, Homological invariants of powers of an ideal, {\it Proc. Am. Math. Soc.} {\bf 118} (1993)
%757--764.

\bibitem{Joellenbeck} M. J\"ollenbeck, On the multigraded Hilbert- and Poincare series of monomial rings,
{\it J. Pure Appl. Algebra.} {\bf 207}, (2006),  261--298.

\bibitem{Lyubeznik} G. Lyubeznik, On the arithmetical rank of monomial ideals, {\it J. Algebra.} {\bf 112}, (1988) 86--89.

\end{thebibliography}
\end{document}